\documentclass[10pt]{amsart}

\usepackage{latexsym,amsmath,amstext,mathrsfs,amsthm,amscd,amsopn,verbatim,amsrefs}

\usepackage{graphicx}
\usepackage[dvips]{color}
\usepackage{epsfig}
\usepackage{psfrag}
\usepackage[all]{xy}
\usepackage{amscd}
\usepackage{amssymb}

\newtheorem{Thm}{Theorem}[section]
\newtheorem{Prop}[Thm]{Proposition}

\theoremstyle{remark}
\newtheorem{Rem}[Thm]{Remark}

\theoremstyle{definition}

\title[A note on the Koszul complex in deformation quantization]{A note on the Koszul complex in deformation quantization}
\author{A.~Ferrario}
\author{C.~A.~Rossi}
\author{T.~Willwacher}
\address{A.~F. and T.~W.: Department of mathematics, ETH Zurich,
8092 Zurich, Switzerland}
\address{C.~A.~R.: Centro de An\'{a}lise Matem\'{a}tica, Geometria e Sistemas Din\^{a}micos,
Departamento de Matem\'atica, Instituto Superior T\'ecnico, Av.
Rovisco Pais, 1049-001 Lisboa, Portugal}

\begin{document}

%\texttt{\tableofcontents}


\begin{abstract}
The aim of this short note is to present a proof of the existence of an $A_\infty$-quasi-isomorphism between the $A_\infty$-$\mathrm S(V^*)$-$\wedge(V)$-bimodule $K$, introduced in~\cite{CFFR}, and the Koszul complex $\mathrm K(V)$ of $\mathrm S(V^*)$, viewed as an $A_\infty$-$\mathrm S(V^*)$-$\wedge(V)$-bimodule, for $V$ a finite-dimensional (complex or real) vector space.
\end{abstract}

\maketitle

\section{Introduction}\label{s-intro}
The main result of~\cite{CFFR} is a Formality Theorem in presence of two subspaces $U_1$, $U_2$ of a real or complex finite-dimensional vector space $V$, constructed using the graphical techniques of Kontsevich~\cite{K}.
Such a Formality Theorem has interesting by-products even in the simplest case, when $U_1=V$ and $U_2=\{0\}$:
\begin{enumerate}
\item It implies that the deformation quantization of $A=\mathcal O_V=\mathrm S(V^*)$~\cite{K} and of $B=\mathcal O_{V^*[1]}=\wedge(V)$~\cite{CF2} w.r.t.\ a given polynomial Poisson structure on $V$ preserves Koszul duality.
\item It yields a proof of a conjecture of B.~Shoikhet~\cites{Sh1,CFR} about the possibility of realizing Kontsevich's deformed algebra $A_\hbar=A[\![\hbar]\!]$~\cite{K} w.r.t.\ a polynomial Poisson structure {\em via} generators and relations.
\end{enumerate} 

In this particular context, the main object appearing in the Formality Theorem is $K=\mathbb K$ (here, $\mathbb K$ is any field of characteristic $0$ containing $\mathbb R$), endowed with a nontrivial $A_\infty$-$A$-$B$-bimodule structure. 
This bimodule structure is constructed explicitly in~\cite[Subsection 6.2]{CFFR}, using Kontsevich's graphical techniques, specializing earlier results of Cattaneo--Felder~\cites{CF,CF2}. 
The main result of this note is the following 
\begin{Thm}\label{t-koszul-bar}
The $A_\infty$-bimodule $K$ is $A_\infty$-quasi-isomorphic to the Koszul complex $\mathrm K(V)$ of $A$, endowed with the usual $A$-$B$-bimodule structure.
\end{Thm}
Theorem~\ref{t-koszul-bar} has been previously stated as a conjecture, see~\cite[Conjecture 1.3]{CFFR}, where the $A_\infty$-$A$-$B$-bimodule has been introduced.

In fact, the Koszul complex, which is a dg $A$-$B$-bimodule, has been considered in the framework of Deformation Quantization, in~\cite{Sh2}, where the author used Tamarkin's Formality in order to quantize both algebras $A$ and $B$ and the Koszul complex itself, in order to prove that Koszul duality is preserved by Deformation Quantization. 

The approach of~\cite{CFFR} makes use of Kontsevich's Formality: the $A_\infty$-$A$-$B$-bimodule structure originates from perturbative expansion of the Poisson Sigma model, and it turns out that it realizes, in the $A_infty$-framework, Koszul duality between $A$ and $B$: it turns out that it behaves well w.r.t.\ Deformation Quantization, hence it is the right candidate for showing that Kontsevich's Deformation Quantization techniques preserve also Koszul duality.

%The conjecture is thus natural, as we have two distinct objects which realize the same duality property between $A$ and $B$.

For the proof we check that both morphisms in the sequence 
\[
\mathrm K(V) \hookrightarrow A\underline\otimes_A K \to K
\]
are quasi-isomorphisms of $A_\infty$-bimodules. Here $(-\underline\otimes_A-)$ denotes a tensor product of $A_\infty$-bimodules, defined in Section~\ref{s-1}. As for ordinary bimodules, one verifies that the tensor product with the algebra itself, i.e., $(A\underline\otimes_A-)$, is (quasi-isomorphic to) the identity, and hence the right arrow is a quasi-isomorphism. The fact that the left arrow is a quasi-isomorphism of $A_\infty$-bimodules is due to a peculiarity of the bimodule structure on $K$, and is proven in Section~\ref{s-3}.
\begin{Rem}\label{r-proof}
It has been pointed out to us by B.~Keller and A.~Khoroshkin independently that Theorem~\ref{t-koszul-bar} should follow from the results of~\cite{CFFR} and general arguments of homological algebra: however, we do not know, for the time being, a rigorous proof in this framework.
\end{Rem}

\subsection*{Acknowledgments} We kindly thank G.~Felder and D.~Calaque for many inspiring discussions, suggestions and constructive criticism.
We also thank B.~Keller and A.~Khoroshkin for the comments in Remark~\ref{r-proof}, although this paper may not answer to their criticism in the way they intended it.
Finally, we thank J.~Stasheff and the referee for having carefully read this paper and for many constructive critics and editorial improvements.
The second author acknowledges the stimulating working atmosphere of the FIM of the ETH, Z\"urich, where this work has been accomplished.

\section{Notation and conventions}\label{s-0}
Throughout the paper, $\mathbb K$ is a field of characteristic $0$, which contains $\mathbb R$. $V$ is a finite-dimensional vector space over $\mathbb K$, $V^*$ its dual.
Further, we denote by $\{x_i\}$, $i=1,\dots,d=\dim V$, a basis of $V^*$, which yields automatically global linear coordinates on $V$, which we denote by ${y_i}$.

Let $\texttt{Mod}_\mathbb K$ be the monoidal category of graded vector spaces, with graded tensor product, and with inner spaces of morphisms (i.e.\ we consider morphisms, which are finite sums of morphisms of any degree); $[\bullet]$ denotes the degree-shifting functor on $\texttt{Mod}_\mathbb K$.
In particular, the identity morphism of an object $M$ of $\texttt{Mod}_\mathbb K$ induces a canonical isomorphism $s:M\to M[1]$ of degree $-1$ with inverse $s^{-1}:M[1]\to M$ (suspension and de-suspension isomorphisms): for the sake of simplicity, we will use the following short-hand notation
\[
(v_1|\cdots|v_n)=s(v_1)\otimes\cdots \otimes s(v_n).
\]
The degree of an element $m$ of a homogeneous component of an object $M$ of $\texttt{Mod}_\mathbb K$ is denoted by $|m|$.
Unadorned tensor products are meant to be over $\mathbb K$.

A (possibly curved) $A_\infty$-algebra $A$ over $\mathbb K$ is equivalent to the structure of a codifferential cofree coalgebra with counit on $\mathrm T(A[1])=\bigoplus_{n\geq 0}A[1]^{\otimes n}$, for an object $A$ of $\texttt{Mod}_\mathbb K$. The codifferential $\mathrm d_A$ is uniquely determined by its Taylor components
\[
\mathrm d_A^n:A[1]^{\otimes n}\to A[1],\ n\geq 0,
\]
all of degree $1$, and the condition that $\mathrm d_A$ squares to $0$ translates into an infinite family of quadratic relations between its Taylor components.
We further set $\mathrm m_A^n=s^{-1}\circ \mathrm d_A^n \circ s^{\otimes n}$. By construction, $\mathrm m_A^n$ are $\mathbb K$-linear maps from $A^{\otimes n}$ to $A$ of degree $2-n$.
We refer to $\mathrm m_A^0$ as to the curvature of $A$. It is an element of $A$ of degree $2$, which measures the failure of $(A,\mathrm m_A^1)$ to be a differential graded (shortly, from now on, dg) vector space over $\mathbb K$.
If $\mathrm m_A^0=0$, then $A$ is said to be flat.

Finally, given two (possibly curved) $A_\infty$-algebras $A$, $B$, an $A_\infty$-$A$-$B$-bimodule structure on an object $K$ of $\texttt{Mod}_\mathbb K$ is equivalent to the structure of a codifferential cofree bicomodule on $\mathrm T(A[1])\otimes K[1]\otimes \mathrm T(B[1])$. As for $A_\infty$-algebras, such a codifferential $\mathrm d_K$ is uniquely determined by its Taylor components
\[
\mathrm d_K^{m,n}:A[1]^{\otimes m}\otimes K[1]\otimes B[1]^{\otimes n}\to K[1],\ m,n\geq 0,
\]
all of degree $1$.
As before, we introduce the maps $\mathrm m_K^{m,n}=s^{-1}\circ \mathrm d_K^{m,n}\circ s^{\otimes m+1+n}$, of degree $1-m-n$.
The condition that $\mathrm d_K$ squares to $0$ is equivalent to an infinite family of quadratic relations between the Taylor components of $\mathrm d_A$, $\mathrm d_B$ and $\mathrm d_K$.
For more details on $A_\infty$-bimodules over (possibly curved) $A_\infty$-algebras, we refer to~\cite[Sections 3 and 4]{CFFR}.
\begin{Rem}\label{r-bimod}
We observe that, if $A$ and $B$ are both flat, then an $A_\infty$-$A$-$B$-bimodule structure on $K$ yields both a left $A_\infty$-$A$- and right $A_\infty$-$B$-module structure on $K$ in the sense of~\cites{Kel,L-H}, but, if either $A$ or $B$ or both are curved, then the $A_\infty$-bimodule structure does not restrict to (left or right) $A_\infty$-module structures, see e.g.~\cite{W} and~\cite[Subsection 4.1]{CFFR}.
\end{Rem}

\section{The tensor product of $A_\infty$-bimodules}\label{s-1}
We consider now three (possibly curved) $A_\infty$-algebras $A$, $B$ and $C$. Furthermore, we consider an $A_\infty$-$A$-$B$-bimodule $K_1$ and an $A_\infty$-$B$-$C$-bimodule $K_2$.
The tensor product of $K_1$ and $K_2$ over $B$, as an element of $\texttt{Mod}_\mathbb K$, is defined as
\[
K_1\underline\otimes_B K_2=K_1\otimes \mathrm T(B[1])\otimes K_2.
\]
The $A_\infty$-structures on $A$, $B$, $C$, $K_1$ and $K_2$ determine a unique structure of $A_\infty$-$A$-$C$-bimodule over $K_1\underline\otimes_B K_2$, which we now describe explicitly.
By the arguments in Section~\ref{s-0}, it suffices to construct a codifferential on the cofree bicomodule
\[
\mathrm T(A[1])\otimes \left(K_1\underline\otimes_B K_2\right)[1]\otimes \mathrm T(C[1])\cong \mathrm T(A[1])\otimes K_1[1]\otimes \mathrm T(B[1])\otimes K_2[1]\otimes \mathrm T(C[1]),
\]
where the isomorphism is induced by suspension and de-suspension, and has degree $-1$; on the latter dg vector space, we define a bicoderivation {\em via}
\begin{equation}\label{eq-bicoder}
\mathrm d_{K_1}\otimes 1\otimes 1+1\otimes 1\otimes \mathrm d_{K_2}-1\otimes 1\otimes\mathrm d_B\otimes 1\otimes 1,
\end{equation}
where $1$ denotes here the identity operator on the corresponding factor.

\begin{Rem}\label{r-susp}
We now need a {\em caveat} regarding the fact that we use suspension and de-suspension in order to construct an $A_\infty$-structure on the tensor product of two $A_\infty$-bimodules, namely, while suspension or de-suspension of an 
$A_\infty$-(bi)module is again an $A_\infty$-(bi)module in a natural way, this is not true anymore for an $A_\infty$-algebra.
In fact, an $A_\infty$-algebra structure on $A$ cannot be transported on $A[-1]$, because of the fact that twisting w.r.t.\ suspension and de-suspension of the Taylor components of the $A_\infty$-structure on $A$ would produce maps of the wrong degree. 
\end{Rem}

It is more instructive to write down explicit formul\ae\ for the Taylor components of the previous bicoderivation
\begin{equation}\label{eq-tayl-tens}
\begin{aligned}
&\mathrm d_{K_1\underline\otimes_B K_2}^{m,n}\!(a_1|\cdots|a_m|k_1\otimes (b_1|\cdots|b_q)\otimes k_2|c_1|\cdots|c_n)=0,\quad m,n>0\\
&\mathrm d_{K_1\underline\otimes_B K_2}^{m,0}\!(a_1|\cdots|a_m|k_1\otimes (b_1|\cdots|b_q)\otimes k_2)=\sum_{l=0}^q s\left(s^{-1}(\mathrm d_{K_1}^{m,l}(a_1|\cdots|a_m|k_1|b_1|\cdots|b_l))\otimes (b_{l+1}|\cdots|b_q)\otimes k_2\right),\quad m>0\\
&\mathrm d_{K_1\underline\otimes_B K_2}^{0,n}\!(k_1\otimes (b_1|\cdots|b_q)\otimes k_2|c_1|\cdots|c_n)=(-1)^{|k_1|+\sum_{j=1}^q(|b_j|-1)}\sum_{l=0}^q s\left(k_1\otimes (b_1|\cdots|b_l)\otimes\right.\\
&\phantom{\mathrm d_{K_1\underline\otimes_B K_2}^{0,n}\!(k_1\otimes (b_1|\cdots|b_q)\otimes k_2|c_1|\cdots|c_n)=}\left.s^{-1}(\mathrm d_{K_2}^{q-l,n}(b_{l+1}|\cdots|b_q|k_2|c_1|\cdots|c_n)\right),\quad n>0,\\
&\mathrm d_{K_1\underline\otimes_B K_2}^{0,0}\!\left(s(k_1\otimes (b_1|\cdots|b_q)\otimes k_2)\right)=\sum_{l=0}^q s\left(s^{-1}(\mathrm d_{K_2}^{0,l}(k_1|b_1|\cdots|b_l)\otimes (b_{l+1}|\cdots|b_q)\otimes k_2\right)+\\
&\phantom{\mathrm d_{K_1\underline\otimes_B K_2}^{0,0}\!\left(s(k_1\otimes (b_1|\cdots|b_q)\otimes k_2)\right)=}+\sum_{0\leq l\leq q\atop 0\leq p\leq q-l} (-1)^{(|k_1|-1)+\sum_{j=1}^l (|b_j|-1)}s\!\left(k_1\otimes (b_1|\cdots|\mathrm d_B^p(b_{l+1}|\cdots|b_{l+p})|\cdots|b_q)\otimes k_2\right)+\\
&\phantom{\mathrm d_{K_1\underline\otimes_B K_2}^{0,0}\!\left(s(k_1\otimes (b_1|\cdots|b_q)\otimes k_2)\right)=}+(-1)^{|k_1|+\sum_{j=1}^q(|b_j|-1)}\sum_{l=0}^q
s\!\left(k_1\otimes (b_1|\cdots|b_l)\otimes s^{-1}(\mathrm d_{K_2}^{q-l,0}(b_{l+1}|\cdots|b_q|k_2)\right).
\end{aligned}
\end{equation}
\begin{Rem}\label{r-signs}
We observe that the signs in~\eqref{eq-tayl-tens} are dictated by Koszul's sign rule, together with the signs arising from the previous isomorphism $\mathrm T(A[1])\otimes \left(K_1\underline\otimes_B K_2\right)[1]\otimes \mathrm T(C[1])\cong \mathrm T(A[1])\otimes K_1[1]\otimes \mathrm T(B[1])\otimes K_2[1]\otimes \mathrm T(C[1])$ of said degree $-1$.
\end{Rem}
\begin{Prop}\label{p-tensor}
For $A_\infty$-algebras $A$, $B$, $C$ and $A_\infty$-bimodules $K_1$, $K_2$ as above, $(K_1\underline\otimes_B K_2,\mathrm d_{K_1\underline\otimes_B K_2})$, where the bicoderivation $\mathrm d_{K_1\underline\otimes_B K_2}$ is defined in~\eqref{eq-tayl-tens}, is an $A_\infty$-$A$-$C$-bimodule.
\end{Prop}
The proof of Proposition~\ref{p-tensor} follows from the following argument: since the bicodifferential $\mathrm d_{K_1\underline\otimes_B K_2}$ is a twist of~\eqref{eq-bicoder}, it suffices to prove that~\eqref{eq-bicoder} squares to $0$.
Using the fact that $\mathrm d_B$, $\mathrm d_{K_1}$ and $\mathrm d_{K_2}$ all square to $0$, we are reduced to prove that
\begin{equation}\label{eq-bicoder-rew}
\begin{aligned}
&(\mathrm d_{K_1}\otimes 1\otimes 1)(1\otimes 1\otimes \mathrm d_{K_2})-(\mathrm d_{K_1}\otimes 1\otimes 1)(1\otimes 1\otimes\mathrm d_B\otimes 1\otimes 1)+\\
&(1\otimes 1\otimes \mathrm d_{K_2})(\mathrm d_{K_1}\otimes 1\otimes 1)-(1\otimes 1\otimes \mathrm d_{K_2})(1\otimes 1\otimes\mathrm d_B\otimes 1\otimes 1)-\\
&-(1\otimes 1\otimes\mathrm d_B\otimes 1\otimes 1)(\mathrm d_{K_1}\otimes 1\otimes 1+1\otimes 1\otimes \mathrm d_{K_2})
\end{aligned}
\end{equation}
vanishes.

Since $\mathrm d_{K_1}$ and $\mathrm d_{K_2}$ are bicoderivations, and $\mathrm d_B$ is a coderivation, the sum of the first and second line in~\eqref{eq-bicoder-rew} can be rewritten as
\[
\begin{aligned}
&1\otimes 1\otimes \mathrm d_B\otimes 1\otimes \mathrm d_C+(1\otimes 1\otimes \mathrm d_B\otimes (\mathrm{pr}_{K_2}\circ\mathrm d_{K_2}))(1\otimes 1\otimes \Delta_B\otimes 1\otimes 1)-\\
&-\mathrm d_A\otimes 1\otimes \mathrm d_B\otimes 1\otimes \mathrm d_A-((\mathrm{pr}_{K_1}\circ\mathrm d_{K_1})\otimes \mathrm d_B\otimes 1\otimes 1)(1\otimes 1\otimes \Delta_B\otimes 1\otimes 1),
\end{aligned}
\]
where we have used the short-hand notation
\begin{equation}\label{eq-short-h}
\mathrm{pr}_{K_1}\circ\mathrm d_{K_1}=(1\otimes (\mathrm{pr}_{K_1}\circ\mathrm d_{K_1})\otimes 1)(\Delta_A\otimes 1\otimes 1),
\end{equation}
$\mathrm{pr}_{K_1}$ denoting the projection from $\mathrm T(A[1])\otimes K_1[1]\otimes \mathrm T(B[1])$ onto $K_1[1]$ and {\em e.g.} $\Delta_A$ is the comultiplication of $\mathrm T(A[1])$; similar notation holds for $\mathrm{pr}_{K_2}\circ\mathrm d_{K_2}$.
The notation $\mathrm{pr}_{K_1}\circ\mathrm d_{K_1}$ means therefore simply that we select only that part of $\mathrm d_{K_1}$ given by the Taylor components $\mathrm d_{K_1}^{m,n}$, thus forgetting about the action of $\mathrm d_B$ and $\mathrm d_C$ on the left and on the right, and similarly for $(\mathrm{pr}_{K_1}\circ\mathrm d_{K_1})$.

Using the same arguments, together with the nilpotence of $\mathrm d_B$, yields the same expression for the last line of~\eqref{eq-bicoder-rew}, up to global minus sign, coming from Koszul's sign rule. 
\begin{Rem}
We point out that, in general, the right-hand side of~\eqref{eq-short-h} should be 
\[
(1\otimes (\mathrm{pr}_{K_1}\circ\mathrm d_{K_1})\otimes 1)(\Delta_A\otimes 1\otimes \Delta_B),
\]
but, in this framework, the second coproduct $\Delta_B$ would be redundant and would produce too many terms, due to the fact that there is already a coproduct $\Delta_B$, whence the reason, why we suppressed it in~\eqref{eq-short-h}.
\end{Rem}
Alternatively, straightforward computations involving quadratic relations w.r.t.\ the Taylor components of $\mathrm d_A$, $\mathrm d_C$ and $\mathrm d_{K_1\underline\otimes_B K_2}$, which in turn can be re-written (after a very tedious book-keeping of all signs involved) in terms of the quadratic relations between the Taylor components of the codifferentials on $A$, $B$, $C$, $K_1$ and $K_2$.
Such computations are similar to the computations in~\cite[Chapter 4]{L-H} in the case of right $A_\infty$-modules.

It is easy to verify that, if both $A$, $C$ are flat, then the Taylor component $m_{K_1\underline\otimes_B K_2}^{0,0}$ yields a structure of dg vector space on $K_1\underline\otimes_B K_2$ (while $B$ may be curved).

\subsection{The $A_\infty$-bar construction of an $A_\infty$-bimodule}\label{ss-1-1}
We consider two (possibly curved) $A_\infty$-algebras, and an $A_\infty$-$A$-$B$-bimodule $K$.
It is easy to verify that e.g.\ $A$ can be endowed with the structure of an $A_\infty$-$A$-$A$-bimodule, whose Taylor components are specified {\em via} the assignment $\mathrm d_A^{m,n}=\mathrm d_A^{m+1+n}$.

In particular, we may form the tensor product $A\underline\otimes_A K$, which has the structure of an $A_\infty$-$A$-$B$-bimodule, according to Proposition~\ref{p-tensor}; similarly, we may consider the $A_\infty$-$A$-$B$-bimodule $K\underline\otimes_B B$.

An $A_\infty$-algebra $A$ is said to be {\bf unital}, if it possesses an element $1$ of degree $0$, such that
\[
\mathrm m_A^2(1\otimes a)=\mathrm m_A^2(a\otimes 1)=a,\quad \mathrm m_A^n(a_1\otimes \cdots\otimes a_n)=0,\ n\neq 2,
\]
if $a_i=1$, for some $i=1,\dots,n$.
If $A$ is unital, and $K$ is an $A_\infty$-$A$-$B$-bimodule, then $K$ is (left-)unital w.r.t.\ $A$, if the identities hold true
\[
\mathrm m_K^{1,0}(1\otimes k)=k,\quad \mathrm m_K^{m,n}(a_1\otimes \cdots\otimes a_m\otimes k\otimes b_1\otimes \cdots\otimes b_n)=0,\ m\neq 1,\ n\geq 0,
\]
if $a_i=1$, for some $i=1,\dots,m$.

Now, there is a natural morphism $\mu$ of $A_\infty$-$A$-$B$-bimodules from $A\underline\otimes_A K$ to $K$: the cofreeness of $A_\infty$-$A$-$B$-bimodules implies that such a morphism is uniquely specified by its Taylor components
\[
\mu^{m,n}:A[1]^{\otimes m}\otimes \left(A\otimes A[1]^{\otimes q}\otimes K\right)[1]\otimes B[1]^{\otimes n}\cong A[1]^{\otimes m+1+q}\otimes K[1]\otimes B[1]^{\otimes n}\to K[1],\quad m,n,q\geq 0,
\]
all of degree $0$, which have to additionally satisfy an infinite family of quadratic relations involving the Taylor components of the $A_\infty$-structures on $A\underline\otimes_A K$ and $K$.
We notice that the isomorphism in the middle is of degree $-1$.

There is a natural candidate for the morphism $\mu$, namely,
\begin{equation}\label{eq-bar-constr}
\begin{aligned}
&\mu^{m,n}(a_1|\cdots|a_m|a\otimes (\widetilde a_1|\cdots|\widetilde a_q)\otimes k|b_1|\cdots|b_n)=\\
&=(-1)^{\sum_{i=1}^m (|a_i|-1)+|a|+\sum_{j=1}^q(|\widetilde a_j|-1)}\mathrm d_K^{m+1+q,n}(a_1|\cdots|a_m|a|\widetilde a_1|\cdots|\widetilde a_q|k|b_1|\cdots|b_n),\quad m,n,q\geq 0.
\end{aligned}
\end{equation}
Similar formul\ae\ hold true for the case of the $A_\infty$-$A$-$B$-bimodule $K\underline\otimes_B B$.
\begin{Prop}\label{p-bar}
For two (possibly curved) $A_\infty$-algebras $A$, $B$ and an $A_\infty$-$A$-$B$-bimodule $K$, there is a natural morphism $\mu$, defined by~\eqref{eq-bar-constr}, of $A_\infty$-$A$-$B$-bimodules from $A\underline\otimes_A K$ to $K$.

If $A$, $B$ are both flat, and $A$, $K$ are (left-)unital, then the $A_\infty$-morphism~\eqref{eq-bar-constr} is an $A_\infty$-quasi-isomorphism.
\end{Prop}
We first observe that the Taylor components in~\eqref{eq-bar-constr} have the right degree.
Then, the proof of Proposition~\ref{p-bar} consists in checking the aforementioned quadratic identities: in view of~\eqref{eq-tayl-tens}, and after a straightforward, but tedious, book-keeping of all signs involved, such quadratic identities can be re-written as the quadratic identities for the Taylor components of the $A_\infty$-structures on $A$ and $K$, whence the first claim follows.

If $A$, $B$ are both flat, then $\mathrm d_{A\underline\otimes_A K}^{0,0}$ endows $A\underline\otimes_A K$ with the structure of a dg vector space: $\mu$ is an $A_\infty$-quasi-isomorphism, if and only if its Taylor component $\mu^{0,0}$ (which is automatically a morphism of dg vector spaces) induces an isomorphism in cohomology.

If $A$ and $K$ are unital, we set 
\begin{align}
\label{eq-inv}\nu^{0,0}(k)&=1\otimes k,\\
\label{eq-homot}\sigma(a\otimes (\widetilde a_1|\cdots|\widetilde a_q)\otimes k)&=1\otimes(a|\widetilde a_1|\cdots|\widetilde a_q)\otimes k,\ q\geq 0.
\end{align}
It is not difficult to check that the degree of~\eqref{eq-inv} is $0$, while the degree of~\eqref{eq-homot} is $-1$.
Moreover, direct computations involving the fourth identity in~\eqref{eq-tayl-tens} and the fact that $A$ and $K$ are unital imply that~\eqref{eq-inv} is homotopically inverse to $\mu^{0,0}$, with explicit homotopy~\eqref{eq-homot}, whence the second claim follows.
\begin{Rem}\label{r-bar}
We observe that Proposition~\ref{p-bar} has been stated and proved in the framework of right $A_\infty$-modules in~\cite{L-H}.
\end{Rem}

\subsection{The $A_\infty$-bimodule structure on the bar resolution of the augmentation module}\label{ss-1-2}
For $V$ as in Section~\ref{s-0}, we consider the symmetric algebra $A=\mathrm S(V^*)$ and the exterior algebra $B=\wedge(V)$: both are unital dg algebras with trivial differential, $A$ is concentrated in degree $0$, while $B$ is non-negatively graded.
In particular, $A$ and $B$ can be viewed as flat, unital $A_\infty$-algebras, whose only non-trivial Taylor components are $\mathrm d_A^2$ and $\mathrm d_B^2$ respectively.

%We may view $K=\mathbb K$ as a left $A$-module and as a right $B$-module {\em via} the natural augmentation.

According to~\cite{CFFR}, $K=\mathbb K$ can be endowed with a non-trivial $A_\infty$-$A$-$B$-bimodule structure, which restricts to the natural augmentation left- and right-modules; non-triviality, here, means that there are non-trivial Taylor components $\mathrm d_K^{m,n}$, for both $m$ and $n$ non-zero, e.g.\
\[
\mathrm m_K^{1,1}(a\otimes 1\otimes b)=\langle b,a \rangle,\quad a\in V,\ b\in V^*,
\]
and $\langle\bullet,\bullet\rangle$ denotes the duality pairing between $V^*$ and $V$.
For a complete description of the $A_\infty$-$A$-$B$-bimodule structure on $K$, we refer to~\cite[Subsection 6.2]{CFFR}.

Since $A$, $B$ are flat, and $A$, $K$ are unital, Proposition~\ref{p-bar} implies that there is an $A_\infty$-quasi-isomorphism of $A_\infty$-$A$-$B$-bimodules from $A\underline\otimes_A K$ to $K$.
A direct computation implies that $A\underline\otimes_A K$ is a dg vector space concentrated in non-positive degrees; recalling~\eqref{eq-tayl-tens}, its $A_\infty$-$A$-$B$-bimodule structure is given by  
\begin{equation}\label{eq-bar-tayl}
\begin{aligned}
\mathrm d_{A\underline\otimes_A K}^{0,0}(a\otimes (\widetilde a_1|\cdots|\widetilde a_q)\otimes 1)&=s\!\left((a\widetilde a_1)\otimes (\widetilde a_2|\cdots|\widetilde a_q)\otimes 1+\sum_{i=1}^{q-1}(-1)^i a\otimes (\widetilde a_1|\cdots|\widetilde a_i\widetilde a_{i+1}|\cdots|\widetilde a_q)\otimes 1+\right.\\
&\phantom{=}\left.+(-1)^q a\otimes (\widetilde a_1|\cdots|\widetilde a_{q-1})\otimes \widetilde a_q(0)\right),\\
\mathrm d_{A\underline\otimes_A K}^{1,0}(a_1|a\otimes (\widetilde a_1|\cdots|\widetilde a_q)\otimes 1)&=s\!\left((a a_1)\otimes (\widetilde a_1|\cdots|\widetilde a_q)\otimes 1\right),\\
\mathrm d_{A\underline\otimes_A K}^{0,n}(a\otimes (\widetilde a_1|\cdots|\widetilde a_q)\otimes 1|b_1|\cdots|b_n)&=(-1)^q\sum_{l=0}^q s\!\left(a_1\otimes (\widetilde a_1|\cdots|\widetilde a_l)\otimes s^{-1}(\mathrm d_K^{q-l,n}(\widetilde a_{l+1}|\cdots|\widetilde a_q|1|b_1|\cdots|b_n))\right),
\end{aligned}
\end{equation}
and in all other cases, the Taylor components are trivial.

The first identity in~\eqref{eq-bar-tayl} yields the identification between the dg vector space $\left(A\underline\otimes_A K,\mathrm d_{A\underline\otimes_A K}^{0,0}\right)$ identifies with the bar complex of the left augmentation module $K=\mathbb K$ over $A$.

Further, the identities~\eqref{eq-bar-tayl} imply that the left $A_\infty$-$A$-module structure on the bar complex $A\underline\otimes _A K$ of $K$ is the standard one, while the non-triviality of the $A_\infty$-$A$-$B$-bimodule structure on $K$ yields non-triviality of the right $A_\infty$-$B$-module structure on $A\underline\otimes_A K$.

\section{The Koszul complex of $A=\mathrm S(V^*)$: a brief {\em memento}}\label{s-2}
For $V$ as in Section~\ref{s-0}, we consider the symmetric algebra $A=\mathrm S(V^*)$ and the exterior algebra $B=\wedge(V)$.% as in Subsection~\ref{ss-1-2}.

We further consider the Koszul complex $\mathrm K(V)$ of $A$. As a graded vector space,
\[
\mathrm K^q(V)=\wedge^{-q}_{\mathcal O_V} \Omega^1_{\mathcal O_V/\mathbb K},\ q\geq 0,
\]
where $\Omega^1_{\mathcal O_V/\mathbb K}$ denotes the module of Kaehler differentials on $\mathcal O_V=\mathrm S(V^*)=A$ as a $\mathbb K$-algebra; in particular, $\mathrm K(V)$ is non-positively graded.
The differential $\partial$ on $\mathrm K(V)$ is induced by the (left) contraction w.r.t.\ the Euler vector field on $V$; further, $\mathrm K(V)$ admits an obvious left $A$-action, and contraction w.r.t.\ polyvector fields on $V$ induces by restriction (and keeping track of Koszul's sign rule) a right $B$-action on $\mathrm K(V)$.
In particular, $\mathrm K(V)$ admits the structure of an $A_\infty$-$A$-$B$-bimodule, whose only non-trivial Taylor components are labeled by pairs of indices $(m,n)$, such that $m+n\leq 1$.
For later computations, we choose a $\mathbb K$-basis $\{x_i\}$ on $V^*$ as in Section~\ref{s-0}: in particular, we may write
\[
\mathrm K(V)\cong \mathbb K[x_i,\theta_j],
\]
where $\{\theta_j\}$ denotes a set of odd coordinates of degree $-1$, which anticommute with each other and commute with $x_i$; w.r.t.\ the previous algebra isomorphism, $x_i\mapsto x_i$, $\mathrm d x_i\mapsto \theta_i$, and $\partial$ is uniquely determined by the graded Leibniz rule (from the left) and by $\partial (x_i)=0$, $\partial (\theta_i)=x_i$.

Furthermore, we may also write $B=\mathbb K[\partial_{\theta_j}]$, where the partial derivative $\partial_{\theta_j}$ has degree $1$, and acts in an obvious way from the left on $\mathrm K(V)$: thus, the right $B$-action on $\mathrm K(V)$ takes the explicit form
\[
\mathrm K(V)\otimes B\ni \eta\otimes b_1\mapsto (-1)^{|\eta||b_1|} b_1(\eta).
\]
We finally observe that $(\mathrm K(V),\partial)$ is a free resolution of the left $A$-module $K$ {\em via} augmentation; it is also a free resolution of the right $B$-module $K$ {\em via} augmentation, because of the isomorphism $B\cong (\mathbb K[\theta_j])[-d]$ induced by contraction w.r.t.\ polyvector fields.

\section{An explicit $A_\infty$-quasi-isomorphism between the bar and the Koszul complex of $A$}\label{s-3}
We use the same notation as in the preceding sections; we only observe that in the whole Section, $A=\mathrm S(V^*)$, $B=\wedge(V)$ and $K=\mathbb K$ with the $A_\infty$-$A$-$B$-bimodule structure from~\cite{CFFR}.

Since $A\underline\otimes_A K$ and $\mathrm K(V)$ are both resolutions of the left augmentation module $K=\mathbb K$ over $A$, arguments from abstract (co)homological algebra imply that they are quasi-isomorphic to each other as complexes of free left $A$-modules.
More precisely, the quasi-isomorphism from $\mathrm K(V)$ to $A\underline\otimes_A K$ as complexes of left $A$-modules has an explicit form, namely
\begin{equation}\label{eq-skew}
\Phi(\theta_{i_1}\cdots\theta_{i_q})=\sum_{\sigma\in\mathfrak S_q}(-1)^\sigma 1\otimes (x_{\sigma(i_1)}|\cdots|x_{\sigma(i_q)})\otimes 1,\quad 1\leq i_1<\cdots< i_q\leq d.
\end{equation}
It suffices to define the morphism $\Phi$ on monomials of the form $\theta_{i_1}\cdots\theta_{i_q}$, and then extend it $A$-linearly on the left.
It follows immediately that $\Phi$ is of degree $0$ and commutes with left $A$-action.
%, recalling the first identity in~\eqref{eq-bar-tayl}.
An easy computation shows that $\Phi$ commutes with differentials. It is a quasi-isomorphism since $\Phi(1)=1$.
%\begin{Rem}\label{r-quasi-inv}
%We observe that $\Phi$ is a quasi-isomorphism because of the existence of an inverse up to homotopy: the ``quasi-inverse" to $\Phi$ is explicitly given by
%\[
%\Psi(1\otimes (\widetilde a_1|\cdots|\widetilde a_q)\otimes 1)=\left(\int_0^1 (\partial_{i_1}\widetilde a_1)(t_1x)\left(\int_0^{t_1} (\partial_{i_2}\widetilde a_2)(t_2x)\left(\cdots \int_0^{t_{q-1}} (\partial_{i_q}\widetilde a_q)(t_qx)\mathrm d t_q\right)\cdots \mathrm d t_2\right)\mathrm d t_1\right)\theta_{i_1}\cdots\theta_{i_q}.
%\]
%\end{Rem}
\begin{Thm}\label{t-koszul-bar-1}
For a finite-dimensional $\mathbb K$-vector space $V$, the morphism~\eqref{eq-skew} extends to a quasi-isomorphism of $A_\infty$-$A$-$B$-bimodules from $\mathrm K(V)$ to $A\underline\otimes_A K$, where the $A_\infty$-$A$-$B$-bimodule structures on $\mathrm K(V)$ and $A\underline\otimes_A K$ are described in Sections~\ref{s-2} and~\ref{ss-1-1} respectively.
%Sections~\ref{s-2} and Subsection~\ref{ss-1-2} respectively.
\end{Thm}
\begin{proof}
We know that~\eqref{eq-skew} is a morphism of degree $0$ from $\mathrm K(V)$ to $A\underline\otimes_A K$: we declare (the conjugation w.r.t.\ $s$ of) $\Phi$ to be the $(0,0)$-th Taylor component of the desired $A_\infty$-quasi-isomorphism, while for $(m,n)$ such that $m+n\geq 1$, we set simply $0$.

By the previous arguments, the only non-trivial identities to check are
\begin{align}
\label{eq-A_inf-1} \mathrm d_{A\underline\otimes_A K}^{0,1}(\Phi(\eta)|b_1)&=\Phi(\mathrm d_{\mathrm K(V)}^{0,1}(\eta|b_1)),\\
\label{eq-A_inf-2} \mathrm d_{A\underline\otimes_A K}^{0,n}(\Phi(\eta)|b_1|\cdots|b_n)&=0,\quad n\geq 2,
\end{align}
for $b_i$, $i=1,\dots,n$, resp.\ $\eta$, a general element of $B$, resp.\ $\mathrm K(V)$.

We begin by proving Identity~\eqref{eq-A_inf-2}: $A$-linearity implies that we may take $\eta$ of the form $\theta_{i_1}\cdots\theta_{i_q}$, $1\leq i_1<\cdots<i_q\leq d$.
Recalling now Identity~\eqref{eq-skew},
% and the second Identity in~\eqref{eq-bar-tayl}
we rewrite the left-hand side in~\eqref{eq-A_inf-2} as
\[
\mathrm d_{A\underline\otimes_A K}^{0,n}(\Phi(\eta)|b_1|\cdots|b_n)=(-1)^q\sum_{l=0}^q\sum_{\sigma\in \mathfrak S_q}(-1)^\sigma s\!\left(1\otimes (x_{\sigma(i_1)}|\cdots|x_{\sigma(i_l)})\otimes s^{-1}(\mathrm d_K^{q-l,n}(x_{\sigma(i_{l+1})}|\cdots|x_{\sigma(i_q)}|1|b_1|\cdots|b_n)\right).
\]
We now analyze the last factor on the right-hand side: degree reasons imply that, for $0\leq l\leq q$,
\[
-(q-l)-1+\sum_{j=1}^n(|b_j|-1)+1\overset{!}=-1\Longleftrightarrow \sum_{j=1}^n |b_j|=n+q-l-1>q-l,
\]
because $n\geq 2$ by assumption.

We recall now from~\cite[Subsection 6.2]{CFFR} that the Taylor components of the $A_\infty$-$A$-$B$-bimodule structure on $K$ are constructed explicitly {\em via} admissible graphs, in a way reminiscent of Kontsevich's graphical technique of~\cite{K}: using the same notation of~\cite{K}, admissible graphs of type $(m,n)$ (elements of $\mathcal G_{m,n}$) are graphs embedded in $\mathbb R\sqcup \mathbb H$ with $m$ vertices of the first type ({\em i.e.} lying in the complex upper half-plane $\mathbb H$), $n$ vertices of the second type ({\em i.e.} $n$ ordered vertices on the real axis $\mathbb R$), and with a certain number of oriented edges between them.
In the present framework, we consider typically elements of $\mathcal G_{0,m+1+n}$, where to the first $m$ vertices of the second type we associate elements of $A$, to the $m+1$-st vertex of the second type $1$ as an element of $\mathbb K$, and to the last $n$ vertices of the second type we associate elements of $B$: accordingly,  oriented edges are associated to elements of $B$, which we view as translation-invariant poly-derivations acting on $A=\mathcal O_V$.
To such admissible graphs are associated polydifferential operators on $A$, $B$ and $K$ with values in $K$ (by the obvious rule that oriented edges correspond to derivatives) and integral weights, for whose precise treatment we refer to~\cite[Subsection 6.2]{CFFR} again: suffice it to recall here that the integral weight of a given admissible graph $\Gamma$ in $\mathcal G_{0,m+1+n}$ is the integral over the compactified configuration space $\mathcal C_{0,m+1+n}^+$ of $m+1+n$ ordered points on the real axis (modulo rescalings and real translations) of a differential form depending explicitly on $\Gamma$, roughly defined {\em via} the rule that a closed $1$-form is associated to an edge connecting two vertices.

The previous strict inequality, which is a consequence of the non-vanishing of the integral weight of any admissible graph appearing in the formula for $\mathrm d_K^{q-l,n}$, implies the claim: in fact, any admissible graph $\Gamma$ in $\mathcal G_{0,n+q-l-1}$ in $\mathrm d_K^{q-l,n}$ must have exactly $n+q-l-1$ arrows departing from the vertices on the right-hand side of the $q-l+1$-st vertex and incoming on the vertices on the left-hand side; no other arrows or loops are allowed.
Since $n\geq 2$ and $A$, $B$ and $K$ are unital, then all vertices (except the $q-l+1$-st vertex) must be at least univalent: more precisely, the vertices on the left-hand side of the $q-l+1$-st vertex must have exactly one incoming arrow (because of degree reasons), while the vertices on the right-hand side must have at least one outgoing arrow, and the previous strict inequality proves that no such graphs exist.

It remains to prove Identity~\eqref{eq-A_inf-1}.
We first evaluate the right-hand side: using the isomorphism at the end of Section~\ref{s-2}, we may write $b_1=\partial_{\theta_{j_1}}\cdots\partial_{\theta_{j_p}}$, whence the right-hand side takes the form $(-1)^{(p+1)q}\Phi(b_1(\eta))$.

The left-hand side of Identity~\eqref{eq-A_inf-1} has the explicit form
\[
\begin{aligned}
\mathrm d_{A\underline\otimes_A K}^{0,1}(\Phi(\eta)|b_1)&=(-1)^q\sum_{l=0}^q\sum_{\sigma\in \mathfrak S_q}(-1)^\sigma s\!\left(1\otimes (x_{\sigma(i_1)}|\cdots|x_{\sigma(i_l)})\otimes s^{-1}(\mathrm d_K^{q-l,1}(x_{\sigma(i_{l+1})}|\cdots|x_{\sigma(i_q)}|1|b_1)\right)=\\
&=(-1)^q\sum_{\sigma\in \mathfrak S_q}(-1)^\sigma s\!\left(1\otimes (x_{\sigma(i_1)}|\cdots|x_{\sigma(i_{q-p})})\otimes s^{-1}(\mathrm d_K^{p,1}(x_{\sigma(i_{q-p+1})}|\cdots|x_{\sigma(i_q)}|1|b_1)\right),
\end{aligned}
\]
where the second equality follows because of degree reasons.

First, if $q\leq p-1$, both sides of Identity~\eqref{eq-A_inf-1} vanish: this follows immediately from the previous formul\ae.
It remains therefore to prove the claim in the case $q\leq p$.

We now take a closer look at the left-hand side of Identity~\eqref{eq-A_inf-1}: we need to understand, in this particular case, the Taylor component $\mathrm d_K^{p,1}$.
We assume $p\geq 1$, because the case $p=0$ follows immediately by direct computations and previous considerations.
In view of~\cite[Subsection 6.2]{CFFR}, $\mathrm d_K^{p,1}$ is a sum over admissible graphs $\Gamma$ in $\mathcal G_{0,p+2}$: in fact, there is only one such admissible graph contributing non-trivially, pictorially
\bigskip
\begin{center}
\resizebox{0.55 \textwidth}{!}{\begin{picture}(0,0)%
\epsfig{file=contr_gr.pstex}%
\end{picture}%
\setlength{\unitlength}{3947sp}%
\begingroup\makeatletter\ifx\SetFigFont\undefined%
\gdef\SetFigFont#1#2#3#4#5{%
  \reset@font\fontsize{#1}{#2pt}%
  \fontfamily{#3}\fontseries{#4}\fontshape{#5}%
  \selectfont}%
\fi\endgroup%
\begin{picture}(8274,3096)(3139,-6491)
\put(6166,-5266){\makebox(0,0)[lb]{\smash{{\SetFigFont{20}{24.0}{\rmdefault}{\mddefault}{\updefault}{\color[rgb]{0,0,0}$\cdots$}%
}}}}
\put(9241,-6406){\makebox(0,0)[lb]{\smash{{\SetFigFont{20}{24.0}{\rmdefault}{\mddefault}{\updefault}{\color[rgb]{0,0,0}$b_1$}%
}}}}
\put(8341,-6406){\makebox(0,0)[lb]{\smash{{\SetFigFont{20}{24.0}{\rmdefault}{\mddefault}{\updefault}{\color[rgb]{0,0,0}$1$}%
}}}}
\put(3886,-6406){\makebox(0,0)[lb]{\smash{{\SetFigFont{20}{24.0}{\rmdefault}{\mddefault}{\updefault}{\color[rgb]{0,0,0}$x_{\sigma(i_{q-p+1})}$}%
}}}}
\put(7306,-6406){\makebox(0,0)[lb]{\smash{{\SetFigFont{20}{24.0}{\rmdefault}{\mddefault}{\updefault}{\color[rgb]{0,0,0}$x_{\sigma(i_q)}$}%
}}}}
\put(5161,-6391){\makebox(0,0)[lb]{\smash{{\SetFigFont{20}{24.0}{\rmdefault}{\mddefault}{\updefault}{\color[rgb]{0,0,0}$x_{\sigma(i_{q-p+2})}$}%
}}}}
\end{picture}%
}\\
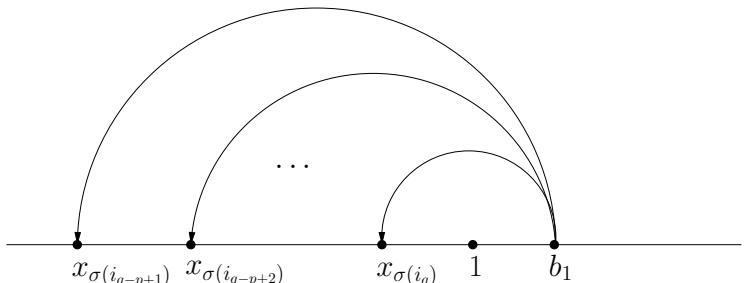
\text{Figure 1 -  The only admissible graph contributing non-trivially to $\mathrm d_K^{p,1}(x_{\sigma(i_{q-p+1})}|\cdots|x_{\sigma(i_q)}|1|b_1)$} \\
\end{center}
\bigskip

The differential operator $\mathcal O_\Gamma$ associated to the admissible graph $\Gamma$ as in Figure 1 can be explicitly evaluated, following the prescriptions in~\cite[Subsection 6.2]{CFFR}, namely
\begin{equation}\label{eq-contr-gr}
s^{-1}(\mathrm d_K^{p,1}(x_{\sigma(i_{q-p+1})}|\cdots|x_{\sigma(i_q)}|1|b_1)=\frac{(-1)^p}{p!} b_1(\theta_{\sigma(i_{q-p+1})}\cdots\theta_{\sigma(i_q)}),
\end{equation}
%The numerical coefficient $(-1)^p/p!$ in front of the right-hand side of Identity~\eqref{eq-contr-gr} is determined by the integral weight of $\mathcal O_\Gamma$, which is computed using the orientation conventions for the (compactified) configuration space $\mathcal C_{0,p+2}^+$ and the $4$-colored propagators from~\cite[Subsubsection 5.3.2]{CFFR}.
%Here, we have used the section of the $G_2=\mathbb R\ltimes \mathbb R^+$-bundle $\mathrm{Conf}_{0,p+2}^+\to \mathcal C_{0,p+2}^+$, which fixes to $0$ the $p+1$-st vertex and to $1$ the $p+2$-nd vertex of the second type, see also~\cite[Subsection 5.2]{CFFR}

The claim now follows from Identity~\eqref{eq-contr-gr}.
\end{proof}

\subsection{A final remark}\label{ss-3-1}
In this final Subsection, we want to point out that the techniques portrayed in Sections~\ref{s-1},~\ref{s-2} and in the present one apply as well to the more general situation examined in~\cite{CFFR}: namely, for any two subspaces $U_i$, $i=1,2$, of a finite-dimensional $\mathbb K$-vector space $V$, we may associate an $A_\infty$-category with two objects, {\em i.e.} $U_1$ and $U_2$, two $A_\infty$-algebras $A$ and $B$ and an $A_\infty$-$A$-$B$-bimodule $K$.

More explicitly, $A$, resp.\ $B$, is the $A_\infty$-algebra associated to the graded vector space of global (regular) sections of the exterior algebra of the normal bundle of $U_1$, resp.\ of $U_2$, in $V$ (with obvious product and trivial differential); $K$ is the graded vector space of global (regular) sections of the exterior algebra of the quotient bundle $TV/(TU_1+TU_2)$ over $U_1\cap U_2$, and the $A_\infty$-bimodule structure extends the natural left $A$- and right $B$-action: of course, $A$, resp.~$B$, resp.~$K$, must be understood as the endomorphism space of $U_1$, resp.\ of $U_2$, resp.\ the space of morphisms from $U_1$ to $U_2$, declaring trivial the remaining one.
We do not indulge in its explicit construction: suffice it to say that it involves (once again) Kontsevich's diagrammatic techniques, and we observe that it reduces, when $U_1=V$ and $U_2=\{0\}$, to the one we have made explicit in the previous computations.

As already remarked in~\cite[Subsection 7.3]{CFFR}, there is also a Koszul complex $\mathrm K(V,U_1,U_2)$ {\em e.g.} for $A$, whose construction is similar to the one of $\mathrm K(V)$ in Section~\ref{s-2}, with obvious due changes; again, $\mathrm K(V,U_1,U_2)$ can be given the structure of a DG $A$-$B$-bimodule (hence, of an $A_\infty$-$A$-$B$-bimodule).

Furthermore, there is a natural (graded) version of the quasi-isomorphism~\eqref{eq-skew} from the Koszul complex $\mathrm K(V,U_1,U_2)$ to the bar resolution of $K$ as a left $A$-module: then, a more involved vanishing lemma, proved in the same spirit of Theorem~\ref{t-koszul-bar-1}, Identity~\eqref{eq-A_inf-1}, and a more refined version of Identity~\eqref{eq-A_inf-2}, imply that such a quasi-isomorphism extends to an $A_\infty$-quasi-isomorphism of $A$-$B$-bimodules.
In particular, Theorem~\ref{t-koszul-bar} holds true in the more general situation of~\cite{CFFR}.

Since Keller's condition in~\cite{Sh2} holds true in the more general situation already examined, {\em i.e.} for a DG category with two objects $U_i$, $i=1,2$, and with spaces of morphisms given by $A$, $B$ (the endomorphisms of $U_1$ and $U_2$ respectively) and $\mathrm K(V,U_1,U_2)$ (the morphisms from $U_1$ to $U_2$), a natural question is, if it is possible to formulate and prove the main result of~\cite{Sh2} in this more general setting, using Tamarkin's results instead of Kontsevich's.

\end{document}